\documentclass[a4paper,11pt,reqno]{amsart}
\usepackage[latin1]{inputenc}
\usepackage{mathrsfs}
\usepackage{dsfont}
\usepackage{hyperref}
\usepackage{amsmath}
\usepackage{amssymb}
\usepackage{amsthm}
\usepackage{amsfonts}
\usepackage{amstext}
\usepackage{amsopn}
\usepackage{amsxtra}
\usepackage{mathrsfs}
\usepackage{dsfont}
\usepackage{esint}
\usepackage{enumitem}

\newtheorem{thm}{Theorem}
\newtheorem{lemma}[thm]{Lemma}
\newtheorem{prop}[thm]{Proposition}

\newtheorem{remark}[thm]{Remark}
\newtheorem{defi}[thm]{Definition}
\newcommand{\R}{\mathbb{R}}
\newcommand{\Z}{\mathbb{Z}}
\newcommand{\C}{\mathbb{C}}

\renewcommand\phi{\varphi}

\newcommand{\cQ}{\mathcal{Q}}
\newcommand{\cV}{\mathcal{V}}

\newcommand{\cD}{\mathcal{D}}

\newcommand{\cL}{\mathcal{L}}
\newcommand{\cI}{\mathcal{I}}

\renewcommand{\geq}{\geqslant}
\renewcommand{\leq}{\leqslant}

\renewcommand{\tilde}{\widetilde}

\newcommand{\be}{\begin{equation}}
\newcommand{\ee}{\end{equation}}
\newcommand{\bq}{\begin{equation}}
\newcommand{\eq}{\end{equation}}

\newcommand{\eps}{\varepsilon}
\usepackage{color}

\title[Multiple solutions for a self-consistent 2d Dirac equation]{Multiple solutions for a self-consistent Dirac equation in two dimensions}

\author[W. Borrelli]{William Borrelli}
\address{Universit\'e Paris-Dauphine, PSL Research University, CNRS, UMR 7534, CEREMADE, F-75016 Paris, France} 
\email{borrelli@ceremade.dauphine.fr}

\date{\today}

\DeclareRobustCommand{\SkipTocEntry}[5]{}

\begin{document}

\begin{abstract}
This paper is devoted to the variational study of an effective model for the electron transport in a graphene sample. We prove the existence of infinitely many stationary solutions for a nonlinear Dirac equation which appears in the WKB limit for the Schr\"{o}dinger equation describing the semi-classical electron dynamics. The interaction term is given by a mean field, self-consistent potential which is the trace of the 3D Coulomb potential. Despite the nonlinearity being 4-homogeneous, compactness issues related to the limiting Sobolev embedding $H^{\frac{1}{2}}(\Omega,\C)\hookrightarrow L^{4}(\Omega,\C)$ are avoided thanks to the regularization property of the operator $(-\Delta)^{-\frac{1}{2}}$. This also allows us to prove smoothness of the solutions. Our proof follows by direct arguments.
\end{abstract}

\maketitle

\medskip

\medskip

\section{Introduction}
New two-dimensional materials possessing Dirac fermions as low-energy excitations have been discovered, the most famous being graphene (see, e.g. \cite{graphene,diracmaterials,introdiracmaterials}), which can be described as a honeycomb lattice of carbon atoms. Those \textit{Dirac materials} possess unique electronic properties which are consequences of the Dirac spectrum. The conical intersection of quasi-particle dispersion relation at degeneracy points leads to an effective massless Dirac equation. A rigorous proof of the existence of Dirac cones in honeycomb structures can be found in \cite{FWhoneycomb}. In this paper we are interested in the mathematical analysis of a model for the quantum transport of electrons in a graphene sample, modeled as a bounded domain in the plane. In \cite{arbunichsparber} the authors gave a rigorous proof of the large, but finite, time-scale validity of a cubic Dirac equation, as a good approximation for the dynamics of the cubic nonlinear Schr\"{o}dinger equation (NLS) with a honeycomb potential, in the weakly nonlinear regime. Their analysis indicates that an analogous result holds for the case of NLS with Hartree nonlinearity. We remark that the linear dynamics has been studied by Fefferman and Weinstein in \cite{FWwaves}.

The (semi-classical) dynamics of electrons in a graphene layer can be described by the following NLS, the interaction being described by a self-consistent potential:
\be\label{nlshartree}
\begin{cases}
i\eps\partial_{t}\Phi^{\eps}=-\eps^{2}\Delta\Phi^{\eps}+V\left(\frac{x}{\eps}\right)\Phi^{\eps}+\eps\kappa\left(\frac{1}{\vert x  \vert}*\vert\Phi^{\eps}\vert^{2} \right)\Phi^{\eps}\\
\Phi^{\eps}(0,x)=\Phi^{\eps}_{0}(x)
\end{cases}
\ee
where $V$ is a honeycomb potential and $\kappa\in\R$ is a coupling constant.

One expects that, as $\eps\rightarrow0$, the dynamics of WKB waves spectrally concentrated around a vertex of the Brillouin zone of the lattice (where the conical degeneracy occurs) can be effectively described by the following Dirac-Hartree equation:
\be\label{dirachartree}
\begin{cases}
i\partial_{t}\varphi=-i\tilde{\sigma}\cdot\nabla\varphi+\kappa\left(\frac{1}{\vert x\vert}*\vert\varphi\vert^{2}\right)\varphi\\
\varphi(0,x)=\varphi_{0}(x)
\end{cases}
\ee
where $\phi:\R_{t}\times\R^{2}_{x}\longrightarrow\C^{2}$, $\tilde{\sigma}:=(\tilde{\sigma}_{1},\tilde{\sigma}_{2})=(\Lambda\sigma_{1},\Lambda\sigma_{2})$, with $\sigma_{i}$ being the first two Pauli matrices 

\begin{equation}\label{pauli} \sigma_{1}:=\begin{pmatrix} 0 \quad& 1 \\ 1 \quad& 0 \end{pmatrix}\quad,\quad \sigma_{2}:=\begin{pmatrix} 0 \quad& -i \\ i \quad& 0 \end{pmatrix} \end{equation}
and
\begin{equation}\label{plambda} \Lambda:=\begin{pmatrix} \overline{\lambda} \quad& 0 \\ 0 \quad& \lambda \end{pmatrix}.\end{equation}
Here $\lambda$ is a constant depending on the potential $V$ (formula 4.1 in \cite{FWhoneycomb}). 

In the above model, the particles can move in all the plane and the potential is the trace on the plane $\{z=0\}$ of the 3-D Coulomb potential. 
\begin{remark}
While the electrons in graphene are essentially confined in 2-D, the electric field clearly still acts in all three spatial dimensions. This justifies the choice of the 3-D Coulomb potential, given by the Riesz potential $(-\Delta)^{-\frac{1}{2}}$. 
\end{remark}
We consider the case where the electrons are constrained to a bounded domain $\Omega\subseteq\R^{2}$ modeling an electronic device. Following El-Hajj and Mehats \cite{mehats}, we define the self-consistent potential using the spectral resolution of $(-\Delta)^{\frac{1}{2}}$ with zero boundary conditions, in order to describe confinement of the electrons. In the same paper, the authors give a formal derivation of the potential $\cV$, starting from the 3-D Poisson equation. Moreover, in \cite{mehats} El-Hajj and Mehats also  prove local well-posedness for two models of electron transport in graphene. More precisely, they treat both the case where $\Omega=\R^{2}$ and $\Omega\subseteq\R^{2}$ is a bounded domain. In the latter case, they replace the Dirac operator by $\sigma_{1}(-\Delta)^{\frac{1}{2}}$, with zero Dirichlet data. It is easy to see that this operator displays a conical band dispersion structure and being off-diagonal it also couples valence and conduction bands, thus mimicking the Dirac operator. More details can be found in the references cited in \cite{mehats}.
\begin{remark}
In the case of graphene the zero-energy for the Dirac operator corresponds to the Fermi level. Then the positive part of the spectrum corresponds to massive conduction electrons, while the negative one to valence electrons. 
\end{remark}
In the present paper we will instead work with the Dirac operator under suitable boundary conditions. It is well-known, in fact, that the Dirac operator, as well as general first order elliptic operators, is not self-adjoint with Dirichlet boundary conditions (see sections 8.2, 10.1 in \cite{reedsimonI} for a counterexemple). 

From now on $\Omega\subseteq\R^{2}$ will denote a smooth bounded open set.

Let $(e_{n})_{n\in\mathbb{N}}\subseteq L^{2}(\Omega)$ be an orthonormal basis of eigenfunctions of the Dirichlet laplacian $(-\Delta)$, with associated eigenvalues $0<\mu_{n}\uparrow+\infty$. 

We define the potential $\cV(\varphi)$ as 
\be\label{potential}
\cV(\varphi):=\sum_{n\geq 0}\mu^{-\frac{1}{2}}_{n}\langle\vert\varphi\vert^{2},e_{n}\rangle e_{n}
\ee
Thus $\cV$ satisfies
\be\label{hartree}
(-\Delta)^{\frac{1}{2}}\cV=\vert\varphi\vert^{2} \qquad\mbox{in} \quad\Omega.
\ee

Working on a bounded domain, we need to choose \textit{local} (for physical reasons) boundary conditions for the Dirac operator. We shall use \textit{infinite mass boundary conditions}, which have been employed in the Physics literature to model quantum dots in graphene (see \cite{terminated} and reference therein). In \cite{infinite} the Dirac operator with infinite mass conditions is proved to be the limit, in the sense of spectral projections, of a Dirac operator with a mass term supported outside $\Omega$, as the mass goes to infinity.


Formally, infinite mass boundary conditions are defined imposing 
\be\label{proj} P\psi:=\frac{1}{2}(\mathds{1}_{2}-\tilde{\sigma}\cdot \boldsymbol{t})\psi=0, \qquad\mbox{on} \quad\partial\Omega
\ee

where $\boldsymbol{t}$ is the tangent to the boundary and $\mathds{1}_{2}$ is the unit matrix. It can be easily seen that such conditions make the Dirac operator $$T:=(-i\tilde{\sigma}\cdot\nabla)$$ symmetric on $L^{2}({\Omega,\C^{2}})$. 

They actually belong to a larger class of local boundary conditions for the Dirac operator (see \cite{terminated}) employed in the theory of graphene, and which are related to M.I.T. and chiral boundary conditions.

\begin{prop}
The unbounded operator $\cD$ formally acting as $T:=(-i\tilde{\sigma}\cdot\nabla)$ on $L^{2}(\Omega,\C^{2})$ is self-adjoint on the domain
\be\label{domain}
D_{\infty}:=\{\psi\in H^{1}(\Omega,\C^{2}) : (P\circ\gamma)\psi=0\}
\ee
where $P$ is the matrix defined in (\ref{proj}) and $\gamma$ is the trace operator.
\end{prop}
The theorem can be proved using the abstract results in \cite{calderon} or following the method of \cite{selfadjoint}, which only requires $C^{2}$-regularity of the boundary.

The model we are going to study thus is given by 
\be\label{diracnewhartree}
\begin{cases}
i\partial_{t}\phi=\cD\phi+\kappa\cV(\phi)\phi,\qquad\mbox{in}\quad\R\times\Omega\\
\phi(0,x)=\phi_{0}(x)
\end{cases}
\ee

Before stating our main theorem, we quickly review the spectral theory for the Dirac operator with infinite-mass boundary conditions.

The compactness of the Sobolev embedding $H^{1}(\Omega,\C^{2})\hookrightarrow L^{2}(\Omega,\C^{2})$ gives that the spectrum of $\cD$ is discrete. Moreover, the domain $D_{\infty}$ is invariant with respect to the antiunitary transformation $\mathcal{U}:=\sigma_{1}\mathcal{C}$, where $\mathcal{C}$ is the complex conjugation on $L^{2}(\Omega,\C^{2})$. Given $\varphi\in D_{\infty}$ we have $$\mathcal{U}\cD\varphi=-\cD\mathcal{U}\varphi. $$
The above observations can be summarized in the following
\begin{prop}
The spectrum $\sigma(\cD)\subseteq\R$ of $\cD$ is purely discrete, symmetric and accumulates at $\pm\infty$.
\end{prop}

Let $(\psi_{k})_{k\in\Z}$ be a Hilbert basis of $L^{2}(\Omega,\C^{2})$ composed of eigenspinors of $\cD$, and $(\lambda_{k})_{k\in\Z}$ the associated eigenvalues, with $\lim_{k\rightarrow\pm\infty}\lambda_{k}=\pm\infty$.

It has been noted in the Physics literature that Dirac operators with infinite mass boundary condition are gapped. A rigorous proof has been recently given in \cite{gapdirac}, where the following result is proved.
\begin{prop}\label{gap}
For any $k\in\Z$ we have $$\lambda^{2}_{k}\geq \frac{2\pi}{\vert\Omega\vert}, $$
where $\vert\Omega\vert$ denotes the area of $\Omega$.
\end{prop}

We look for stationary solutions to the equation (\ref{diracnewhartree}), that is, of the form $$ \phi(t,x)=e^{-i\omega t}\psi(x).$$
Plugging it into the equation one gets
\be\label{stationary}
(\cD-\omega)\psi+\kappa\cV(\psi)\psi=0
\ee
Our main result is the following:
\begin{thm}\label{main}
Fix $\omega\notin\sigma(\cD)$. Then equation (\ref{stationary}) admits infinitely many solutions in $C^{\infty}(\Omega,\C^{2})$ satisfying the boundary condition (\ref{proj}).
\end{thm}
We remark that a variational proof of existence and multiplicity for 3D Maxwell-Dirac and Dirac-Coulomb equations can be found in \cite{maxwelldirac}. Those results have been improved in \cite{Abenda}. The case of subritical Dirac equations on compact spin manifolds has been treated in \cite{Isobe}, for nonlinearity with polynomial growth, and using a Galerkin-type approximation. Our proof is variational and based on direct arguments. The present work has been inspired by the above mentioned articles and by the papers \cite{arbunichsparber,mehats}. 

For the sake of simplicity we will restrict ourselves to $\omega\in(-\lambda_{1},\lambda_{1})$.
\begin{remark}
Without loss of generality, we can choose $\kappa<0$. In particular, we take $\kappa=-1$. The case $\kappa>0$ follows considering the functional $$\cL(\psi):=-\cI(\psi)$$ (see below).
\end{remark}
Solutions to (\ref{stationary}) will be obtained as critical points of the functional
\be\label{action}
\cI(\psi)=\frac{1}{2}\int_{\Omega}\langle(\cD-\omega)\psi,\psi\rangle - \frac{1}{4}\int_{\Omega}\cV(\psi)\vert\psi\vert^{2}
\ee
which is defined and of class $C^{2}$ on the Hilbert space
\be\label{Hilbert}
X:=\left\{\psi\in L^{2}(\Omega,\C^{2}): \Vert\psi\Vert^{2}_{X}:=\sum_{k\in\Z}\vert\lambda_{k}-\omega\vert\vert\langle\psi,\psi_{k}\rangle\vert^{2}<\infty\right\}
\ee
endowed with the scalar product $$\langle\phi,\psi\rangle_{X}:=\langle\phi,\psi\rangle_{L^{2}}+\sum_{k\in\Z}\vert\lambda_{k}-\omega\vert\langle\phi,\psi_{k}\rangle_{L^{2}}\overline{\langle\psi,\psi_{k}\rangle}_{L^{2}}. $$

In the proof of Theorem 1 in \cite{selfadjoint} it is proved that for some $C>0$ there holds
\be\label{equivalent}
\Vert \varphi\Vert_{H^{1}}\leq C (\Vert \varphi\Vert_{L^{2}}+\Vert T\varphi\Vert_{L^{2}})
\ee
for all $\varphi\in D_{\infty}$, that is, for spinors in the operator domain.

This implies that the $H^{1}$-norm and the quantity in brackets in the r.h.s. of (\ref{equivalent}) are equivalent $D_{\infty}$. Then interpolating between $L^{2}(\Omega,\C^{2})$ and $D_{\infty}$, one gets the following
\begin{remark}
The $X$-norm above defined and the $H^{\frac{1}{2}}$-norm are equivalent on the space $X$. This will be repeatedly used in the sequel in connection with Sobolev embeddings.
\end{remark}
We can thus split $X$ as the direct sum of the positive and the negative spectral subspaces of $(\cD-\omega)$:
\be\label{subspaces}
X = X^{+}\oplus X^{-}
\ee
Accordingly we will write $\psi=\psi^{+}+\psi^{-}$. 

The functional (\ref{action}) then takes the form
\be\label{action2}
 \cI(\psi)=\frac{1}{2}\left(\Vert\psi^{+}\Vert^{2}_{X} - \Vert\psi^{-}\Vert^{2}_{X}\right) - \frac{1}{4}\int_{\Omega}\cV(\psi)\vert\psi\vert^{2}
 \ee
Smoothness of the solutions will follow by standard bootstrap arguments.
\begin{remark}
Despite the term in (\ref{action}) involving the potential being 4-homogeneous we can take advantage of regularization property of $(-\Delta)^{-\frac{1}{2}}$, thus avoiding compactness issues related to the limiting Sobolev embedding $X\hookrightarrow L^{4}$. This is in constrast with \cite{shooting}, where we studied the case of a Kerr-like, cubic nonlinearity. We dealt with the lack of compactness through a suitable radial ansatz, reducing the proof to dynamical systems arguments.
\end{remark}

In the sequel, we will denote $X$-norm and the $L^{p}$-norm of a spinor $\psi$ by $\Vert\psi\Vert$ and $\Vert\psi\Vert_{p}$, respectively. Occasionally, we will also omit the domain of definition of functions, denoting $L^{p}$ and Sobolev spaces.
\section{The variational argument}

This section is devoted to the proof of our main theorem. The strategy consists in exploiting the $\mathbb{Z}_{2}$-symmetry of the functional using topological arguments, in order to get multiple solutions. Our argument proceeds suitably splitting the Hilbert space $X$ according to the spectral decomposition of the operator $(\cD-\omega)$. This allows us to define an increasing sequences of critical values $c_{j}\uparrow+\infty$ for the action functional (\ref{action}). Compactness of critical sequences is proved exploiting the regularizing effect of $(-\Delta)^{-\frac{1}{2}}$.

It is easy to see that the functional $\cI$ is even :
$$\cI(-\psi)=\cI(\psi), \quad \forall\psi\in X $$
and this allows us to prove a multiplicity result using a straightforward generalization of the fountain theorem, well-known for semi-definite functionals (see, e.g. \cite{willem}). It in turn relies on the following infinite-dimensional Borsuk-Ulam theorem \cite{rothe}.

Let $H$ be a Hilbert space. 

\begin{defi}
We say that $\Phi:H\longrightarrow H$ is a \textit{Leray-Schauder map} (LS-map) if it is of the form
\be\label{ls}
\Phi = I+K
\ee
where $I$ is the identity and $K$ is a compact operator.
\end{defi}
\begin{thm}{(Borsuk-Ulam in Hilbert spaces)}\label{borsukulam}
Let $Y\leq H$ be a codimension one subspace of $H$ and $\mathcal{U}$ be a symmetric (i.e. $\mathcal{U}=-\mathcal{U}$) bounded neighborhood of the origin. If $\Phi:\partial\mathcal{U}\longrightarrow Y $ is an odd LS-map, then there exists $x\in\partial\mathcal{U}$ such that $\Phi(x)=0$.
\end{thm}
The proof of the above theorem is achieved approximating the compact map by finite rank operators and using the finite-dimensional Borsuk-Ulam theorem, as shown in \cite{rothe}.

Consider an Hilbert basis $(e_{k})_{k\in\Z}$ of $H$. For $j\in\Z$ we define
\be
H_{1}(j):=\overline{span\{e_{k}\}}^{k=j}_{-\infty},\quad H_{2}(j):=\overline{span\{e_{k}\}}^{+\infty}_{k=j}
\ee

Given $0<r_{j}<\rho_{j}$ we set
$$B(j):=\{\psi\in\mathbf{H}_{1}(j)\;:\;\Vert\psi\Vert\leq\rho_{j}\} $$
$$S(j):=\{\psi\in\mathbf{H}_{1}(j)\;:\;\Vert\psi\Vert=\rho_{j}\} $$
$$N(j):=\{\psi\in\mathbf{H}_{2}(j)\;:\;\Vert\psi\Vert=r_{j}\} $$

Let $\mathcal{L}\in C^{1}(H,\R)$ be an even functional of the form 
\be\label{standardform}
\cL(\psi)=\frac{1}{2}\langle L\psi,\psi\rangle+F(\psi)
\ee
where $$L:H_{1}(j)\oplus H_{2}(j)\longrightarrow H_{1}(j)\oplus H_{2}(j)$$ is linear, bounded and self-adjoint and $dF$ is a compact map.

It is a well-known result (see,e.g. (\cite{rabinowitz}, Appendix) and \cite{struwevariational}) that such a functional admits an odd pseudo-gradient flow of the form 

\be\label{pgflow}
\eta(t,*)=\Lambda(t,*)+K(t,*)=\Lambda_{1,j}(t,*)\oplus\Lambda_{2,j}(t,*)+K(t,*),
\ee

where $\Lambda_{i,j}(t,*):\mathbf{H}_{i}(j)\rightarrow\mathbf{H}_{i}(j)$ is an isomorphism ($i=1,2$), and $K(t,*)$ is a compact map.

\begin{thm}{(Fountain theorem)}\label{fountain}
With the above notations, define the min-max level
\be\label{minmax} 
c_{j}:=\inf_{\gamma\in\Gamma(j)}\sup \cL(\gamma(1,B(j)))
\ee
where $\Gamma(j)$ is the class of maps $\gamma\in C^{0}([0,1]\times B(j),H)$ such that $$\gamma(t,\psi)=\psi,\qquad\forall (t,\psi)\in [0,1]\times S(j) $$ and which are homotopic to the identity through a family of odd maps of the form (\ref{pgflow}).
If there holds 
\be\label{inequality}
\inf_{\psi\in N(j)}\mathcal{L}(\psi)=:b_{j}>a_{j}:=\sup_{\psi\in S(j)}\cL(\psi),
\ee
then $c_{j}\geq b_{j}$ and there exists a Cerami sequence $(\psi^{j}_{n})_{n\in\mathbb{N}}\subseteq H$, that is 
\begin{equation}\label{cerami}
 \begin{cases}
\cL(\psi^{j}_{n})\longrightarrow c_{j}\\
(1+\Vert\psi^{j}_{n}\Vert)d\cL(\psi^{j}_{n})\xrightarrow{H^{*}} 0 \qquad\mbox{as}\quad n\longrightarrow\infty
\end{cases}
\end{equation}
where $H^{*}$ is the dual space of $H$.

Moreover, if Cerami sequences are pre-compact, then $c_{j}$ is a critical value.
\end{thm}
\begin{proof}
The proof follows by a standard deformation argument (see \cite{rabinowitz, struwevariational}).  However, we quickly sketch the proof for the convenience of the reader. More details can be found in the mentioned references.

Fix $j\in\mathbb{N}$. We first show that 
\be\label{maggiore}
c_{j}\geq b_{j}>a_{j},
\ee
 where $c_{j}$ is the min-max value defined in (\ref{minmax}) and $b_{j}$ is as in (\ref{inequality}). To this aim, we need to prove the intersection property $$\gamma(1,B(j))\cap N(j)\neq\emptyset$$
for any $\gamma\in\Gamma(j)$. Since $\gamma$ is odd in the second variable, $\gamma(1,0)=0$ and the set
\be
\mathcal{U}=\{u\in B(j) : \Vert\gamma(1,u)\Vert<r\} 
\ee

is a bounded neighborhood of the origin such that $-\mathcal{U}=\mathcal{U}$. 

Let $P:H\longrightarrow Y:= \overline{span\{e_{k}\}}^{k=j-1}_{-\infty}$ be the projection. Consider the map $(P\circ \gamma)(1,*): \partial\mathcal{U}\longrightarrow Y$. We have to prove that the equation 
\be\label{intersection} 
(P\circ \gamma)(1,u)=0
\ee
admits a solution $u_{0}\in\partial\mathcal{U}$.

Recall that $ \gamma(1,*)$ is of the form (\ref{pgflow}). Then (\ref{intersection}) is equivalent to $$ u+(\underbrace{P\circ\Lambda^{-1}(1,*)\circ K)}_{\text{compact}}(1,u)=0$$
and the claim follows by the Borsuk-Ulam theorem (Theorem \ref{borsukulam}).

 We claim that there is a \textit{Palais-Smale sequence} at level $c_{j}$ (PS$_{c_{j}}$ sequence, for short), that is, there exists a sequence $(\psi^{j}_{n})_{n\in\mathbb{N}}\subseteq H$ such that there holds
\begin{equation}\label{ps}
 \begin{cases}
\cL(\psi^{j}_{n})\longrightarrow c_{j}\\
 d\cL(\psi^{j}_{n})\xrightarrow{H^{*}} 0 \qquad\mbox{as}\quad n\longrightarrow\infty.
\end{cases}
\end{equation}
If this is not the case, since $\cL$ is of class $C^{2}$ this implies that there exist $\delta,\varepsilon>0$ such that 
\be\label{farfromzero}
\Vert d\cL(\psi)\Vert \geq \delta>0\ee 
for $\psi\in\{c_{j}-\varepsilon\leq\cL\leq c_{j}+\varepsilon \}.$

By (\ref{minmax}) it follows that there exists $\gamma_{\varepsilon}\in\Gamma(j)$ such that $$\sup \cL(\gamma_{\varepsilon}(1,B(j)))\leq c_{j}+\varepsilon.$$

 Following the construction explained, for instance, in \cite{rabinowitz, struwevariational}, one can construct a suitable vector field (a \textit{pseudo-gradient vector field} for $\cL$) whose flow is as in (\ref{pgflow}) and such that $\frac{d}{dt}\cL(\eta(t,\psi))\leq0$, $\forall (t,\psi)\in [0,1]\times H$. Moreover, combining (\ref{maggiore},\ref{farfromzero}) and choosing $\varepsilon>0$ small, one can also obtain $\eta(t,\psi)=\psi$, for $\psi\notin\{ c_{j}-\varepsilon\leq\cL\leq c_{j}+\varepsilon\}$, and that $\eta(1,\cdot)$ maps $\{\cL\leq c_{j}+\varepsilon \}$ to $\{\cL\leq c_{j}-\varepsilon \}$. Combining those observations one gets that $\eta\circ\gamma_{\epsilon}\in\Gamma(j)$. But then
 \be
 \sup\cL(\eta\circ\gamma_{\epsilon}(1,B(j)))\leq c_{j}-\varepsilon,
 \ee
contradicting the definition of the min-max value (\ref{minmax}). This proves the existence of a PS$_{c_{j}}$-sequence $(\psi^{j}_{n})_{n}$. Moreover, applying Ekeland's variational principle \cite{ekeland} this can be promoted to a Cerami sequence, thus concluding the proof. 
\end{proof}

Our aim is to apply the fountain theorem to the functional $\cI$. First of all, we need to study the geometry of the functional $\cI$.
\begin{prop}\label{form}
The functional $$ \cI(\psi)=\frac{1}{2}\int_{\Omega}\langle(\cD-\omega)\psi,\psi\rangle -\frac{1}{4}\int_{\Omega}\cV(\psi)\vert\psi\vert^{2}$$ is of the form (\ref{standardform}). 
\end{prop}
\begin{proof}
The term involving the potential is 4-homogeneous, but we can avoid compactness issues related to the critical Sobolev embedding $H^{\frac{1}{2}}(\Omega,\C^{2})\hookrightarrow L^{4}(\Omega,\C^{2})$ thanks to the regularizing properties of $(-\Delta)^{-\frac{1}{2}}$, as shown in Proposition (\ref{homogeneous}).
\end{proof}

For each $j\geq1$, consider the splitting
\be
X=X_{1}(j)\oplus X_{2}(j)=\left({\overline{span\{\psi_{k}\}}^{j}_{k=-\infty}}\right)\oplus\left(\overline{span\{\psi_{k}\}}^{+\infty}_{k=j}\right)
\ee
where $(\psi_{k})_{k\in\Z}$ is an orthonormal basis of eigenspinors of $\cD$.
\begin{lemma}
Let $j\geq1$, there exists $\rho_{j}>0$ such that $\cI(\psi)\leq0$, for $\psi\in X_{1}(j)$ and $\Vert\psi\Vert\geq\rho_{j}$.
\end{lemma}
\begin{proof}
Let $\psi\in X_{1}(j)$ be such that $\Vert\psi\Vert\geq\rho_{j}>0$. Recall that $$\psi=\psi^{-}+\psi^{+}\in Y:=X^{-}\oplus span\{e_{k}\}^{k=1}_{j}.$$

Suppose that
\be
\Vert\psi^{-}\Vert\geq\Vert\psi^{+}\Vert.
\ee
It is immediate from (\ref{action2}) that $\cI(\psi)\leq0$.

Now assume 
\be\label{assumption}
\Vert\psi^{+}\Vert\geq\Vert\psi^{-}\Vert.
\ee

We claim that there exists $C=C(j)>0$ such that 
\be\label{claim}
\cQ(\psi):=\int_{\Omega}\cV(\psi)\vert\psi\vert^{2}\geq C\Vert\psi\Vert^{4}
\ee
for all $\psi\in Y$ satisfying (\ref{assumption}).

Suppose the claim is false. Then arguing by contradiction and by the 4-homogeneity of $\cQ$, there exists a sequence $(\psi_{n})_{n\in\mathbb{N}}\subseteq Y$ satisfying (\ref{assumption}), and such that $\Vert\psi_{n}\Vert=1$ and 
$$\cQ(\psi_{n})\longrightarrow 0,\qquad \mbox{as}\quad n\rightarrow+\infty. $$
Notice that (\ref{assumption}) implies that 
\be\label{nonzero2}
\Vert\psi^{+}_{n}\Vert\geq\frac{1}{\sqrt{2}}
\ee
Up to subsequences, we can assume that there exists $\psi_{\infty}\in Y$ such that $\psi^{-}_{n}$ weakly converges to $\psi^{-}_{\infty}$, while $\psi^{+}_{n}$ strongly converges to $\psi^{+}_{\infty}$, the latter sequence lying in a finite-dimensional space. 
Thus there holds
\be\label{nonzero3}
\Vert\psi^{+}_{\infty}\Vert\geq\frac{1}{\sqrt{2}}.
\ee
Since $\cQ$ is continuous and convex it also is weakly lower semi-continuous, and then 
$$ \cQ(\psi_{\infty})=0.$$
This implies that $$\psi_{\infty}=\psi^{+}_{\infty}+\psi^{-}_{\infty}=0 $$ and thus 
\be\label{zero}
\psi^{+}_{\infty}=0
\ee
$\psi^{-}_{\infty}$ and $\psi^{+}_{\infty}$ being orthogonal, contradicting (\ref{nonzero3}). 

Then , given (\ref{claim}), we have
\be
\cI(\psi)\leq\Vert\psi^{+}\Vert-\Vert\psi^{-}\Vert-C\Vert\psi\Vert^{4}
\ee
for all $\psi\in Y$ such that (\ref{assumption}) holds. Thus $\cI(\psi)\leq0$, for $\rho_{j}>0$ large enough. 

\end{proof}

\begin{lemma}\label{diverge}
For $1\leq p<4$ define
$$\beta_{j,p}:=\sup\{\Vert\psi\Vert_{p} : \psi\in X_{2}(j),\Vert\psi\Vert=1\}. $$
Then $\beta_{j,p}\longrightarrow 0 $ as $j\rightarrow\infty$.
\end{lemma}
\begin{proof}
By definition, for each $j\geq 1$ there exists $\psi_{j}\in X_{2}(j)$ such that $\Vert\psi_{j}\Vert=1$ and $\frac{1}{2}\beta_{j,p}<\Vert\psi_{j}\Vert_{p}$. The compactness of the Sobolev embedding implies that, up to subsequences, $\psi_{j}\rightharpoonup\psi$ weakly in $X$ and $\psi_{j}\longrightarrow\psi$ strongly in $L^{p}(\Omega,\C^{2})$. It is evident that $\psi=0$. Then
$$ \frac{1}{2}\beta_{j,p}<\Vert\psi_{j}\Vert_{p}\longrightarrow0.$$
\end{proof}
The above result allows us to prove the following:
\begin{lemma}
There exists $r_{j}>0$ such that 
$$b_{j}:=\inf \{\cI(\psi) : \psi\in X_{2}(j), \Vert\psi\Vert=r_{j}\}\longrightarrow+\infty$$
as $j\longrightarrow+\infty$.
\end{lemma}
\begin{proof}
By the H\"{o}lder inequality, we get
\be\label{bound}
\int_{\Omega}\cV(\psi)\vert\psi\vert^{2}\leq\left(\int_{\Omega}\vert\psi\vert^{3}\right)^{\frac{2}{3}}\left(\int_{\Omega}\cV(\psi)^{3}\right)^{\frac{1}{3}}\leq C\Vert\psi\Vert^{4}_{3}.
\ee
Recall that $\cV(\psi):=(-\Delta)^{-\frac{1}{2}}(\vert\psi\vert^{2})$. Since $\vert\psi\vert^{2}\in L^{\frac{3}{2}}(\Omega,C^{2})$, and $(-\Delta)^{-\frac{1}{2}}$ sends $L^{\frac{3}{2}}(\Omega,\C^{2})$ into $W^{1,\frac{3}{2}}(\Omega,\C^{2})\hookrightarrow L^{3}(\Omega,\C^{2})$, we easily get (\ref{bound}).  

Take $\psi\in X_{2}(j)$ such that $\Vert\psi\Vert=r$. Then by (\ref{bound}) and Lemma \ref{diverge} we have

\be\label{control}
\begin{split}
\cI(\psi)&=\frac{1}{2}\int_{\Omega}\langle(\cD-\omega)\psi,\psi\rangle-\int_{\Omega}\cV(\psi)\vert\psi\vert^{2} \\ &\geq\frac{1}{2}\Vert\psi\Vert^{2}-\frac{1}{2}\Vert\psi\Vert^{2}_{2}-C\Vert\psi\Vert^{4}_{3}\\ &\geq \frac{1}{2}r^{2}-\frac{1}{2}r^{2}\beta^{2}_{j,2}-Cr^{4}\beta^{4}_{j,3}\\&\geq\frac{1}{4}r^{2}-C\beta^{4}_{j,3}r^{4}
\end{split}
\ee
where we used the fact that $\beta^{2}_{j,2}\leq\frac{1}{2}$. 

The function $r\mapsto \frac{1}{4}r^{2}-C\beta^{4}_{j,3}r^{4}$ attains its maximum at $r = (8C\beta^{4}_{j,3})^{-\frac{1}{2}}$. Then taking $r_{j}:=(8C\beta^{4}_{j,3})^{-\frac{1}{2}}$ we get $$b(j)\geq (64C\beta^{3}_{j,3})^{-1}\longrightarrow +\infty $$
and this concludes the proof.
\end{proof}
The above results allow us to apply the Fountain theorem (Theorem \ref{fountain}) to the functional $\cI$. We thus get the existence of a sequence of min-max values 
\be
c_{j}\longrightarrow+\infty, \qquad \mbox{as}\quad  j\rightarrow+\infty,
\ee

and, for each $j\in\mathbb{N}$, of a Cerami sequence $(\psi^{n}_{j})_{n\in\mathbb{N}}\in X$:
\be
\begin{cases}
\cI(\psi^{n}_{j})\longrightarrow c_{j}\\
(1+\Vert\psi^{n}_{j}\Vert)d\cI(\psi^{n}_{j})\xrightarrow{X^{*}} 0 \qquad\mbox{as}\quad n\longrightarrow\infty
\end{cases}
\ee

\begin{lemma}\label{precompact}
Cerami sequences for $\cI$ are pre-compact .
\end{lemma}
\begin{proof}
Let $(\psi_{n})\subseteq X$ be an arbitrary Cerami sequence for $\cI$. 

Then
\be\label{cerami}
\begin{cases}
\cI(\psi_{n})\longrightarrow c\\
(1+\Vert\psi_{n}\Vert)(\cD\psi_{n}-\cV(\psi_{n})\psi_{n})\xrightarrow{X^{*}} 0 \qquad\mbox{as}\quad n\longrightarrow\infty
\end{cases}
\ee
for some $c>0$.

The second condition in (\ref{cerami}) implies that 
\be\label{strong}
\int_{\Omega}\langle\cD\psi_{n},\psi_{n}\rangle-\int_{\Omega}\cV(\psi_{n})\vert\psi_{n}\vert^{2}\longrightarrow 0.
\ee
Combining (\ref{strong}) and the first line in (\ref{cerami}) one gets 
\be\label{onehalf}
\Vert\cV(\psi_{n})\Vert^{2}_{\mathring{H}^{\frac{1}{2}}}=\int_{\Omega}\vert(-\Delta)^{\frac{1}{4}}\cV(\psi_{n})\vert^{2}=\int_{\Omega}\cV(\psi_{n})\vert\psi_{n}\vert^{2}\longrightarrow 2c.
\ee
By the Sobolev embedding $(\cV(\psi_{n}))_{n\in\mathbb{N}}$ is thus bounded in $L^{4}$. Moreover, since $(-\Delta)^{-\frac{1}{2}}$ is positivity-preserving (see section \ref{appendix}), (\ref{onehalf}) implies that $(\cV(\psi_{n})\vert\psi_{n}\vert^{2})_{n\in\mathbb{N}}$ is bounded in $L^{1}$. 

By the above remarks, writing 
\be\label{product}
\cV(\psi_{n})\vert\psi_{n}\vert=\underbrace{\left(\cV(\psi_{n})\vert\psi_{n}\vert^{2}\right)^{\frac{1}{2}}}_{L^{2}-bounded}\underbrace{\left(\cV(\psi_{n})\right)^{\frac{1}{2}}}_{L^{8}-bounded} 
\ee
and by the H\"{o}lder inequality, we easily get that $(\cV(\psi_{n})\psi_{n})_{n\in\mathbb{N}}$ is bounded in $L^{\frac{8}{5}}$. The second line of (\ref{cerami}) gives 
\be\label{bootstrap}
\psi_{n}=\psi^{1}_{n}+\psi^{2}_{n}:=(\cD-\omega)^{-1}(\cV(\psi_{n})\psi_{n})+o(1),\qquad\mbox{in}\quad H^{\frac{1}{2}}(\Omega,\C^{2})
\ee
It is immediate to see that $(\psi^{1}_{n})_{n\in\mathbb{N}}$ is bounded in $W^{1,\frac{8}{5}}\hookrightarrow H^{\frac{1}{2}}$, and thus $(\psi_{n})_{n\in\mathbb{N}}$ is bounded in $H^{\frac{1}{2}}$. 

Up to subsequences, there exists $\psi_{\infty}\in X$ such that $\psi_{n}\rightharpoonup\psi_{\infty}$ weakly in $X$ and $\psi_{n}\rightarrow\psi_{\infty}$ strongly in $L^{p}$ for all $1\leq p<4$.

Since $(\psi_{n})_{n\in\mathbb{N}}$ is a Cerami sequence, there holds
\be\label{opiccolo}
o(1)=\langle d\cI(\psi_{n}),\psi^{+}_{n}-\psi^{+}_{\infty}\rangle=\int_{\Omega}\langle\cD\psi^{+}_{n},\psi^{+}_{n}-\psi^{+}_{\infty}\rangle-\int_{\Omega}\cV(\psi_{n})\langle\psi_{n},\psi^{+}_{n}-\psi^{+}_{\infty}\rangle.
\ee
Moreover, the H\"{o}lder inequality gives
\be\label{opiccolo2}
\begin{split}
\left\vert\int_{\Omega}\cV(\psi_{n})\langle\psi_{n},\psi^{+}_{n}-\psi^{+}_{n}\rangle \right\vert&\leq\int_{\Omega}\cV(\psi_{n})\vert\psi_{n}\vert\vert\psi^{+}_{n}-\psi^{+}_{n}\vert \\ &\leq\Vert\cV(\psi_{n})\vert\psi_{n}\vert\Vert_{2}\Vert\psi^{+}_{n}-\psi^{+}_{\infty}\Vert_{2}\\ &\leq\Vert\cV(\psi_{n})\Vert_{6}\Vert\psi_{n}\Vert_{3}\Vert\psi^{+}_{n}-\psi^{+}_{n}\Vert_{2}\\&\leq C \Vert\psi^{+}_{n}-\psi^{+}_{n}\Vert_{2}
\end{split}
\ee
where in the last line we used the fact that $(\psi_{n})$ is bounded in $H^{\frac{1}{2}}\hookrightarrow L^{3}$ and that $(\cV(\psi_{n}))_{n\in\mathbb{N}}$ is $L^{6}$-bounded. Combining (\ref{opiccolo}) and (\ref{opiccolo2}) we get
\be\label{opiccolo3}
\int_{\Omega}\langle\cD\psi^{+}_{n},\psi^{+}_{n}-\psi^{+}_{\infty}\rangle=o(1)
\ee
On the other hand, for any $\eta^{+}\in X^{+}$, there holds
\be\label{eigbound}
\int_{\Omega}\langle\cD\eta^{+},\eta^{+}\rangle\geq(1+(\lambda_{1})^{-1})\Vert\eta^{+}\Vert^{2}
\ee
as it can be easily checked. By (\ref{eigbound}) and (\ref{opiccolo3}) we thus obtain
\be
\Vert\psi^{+}_{n}-\psi^{+}_{\infty}\Vert^{2}\leq C\int_{\Omega}\langle\cD(\psi^{+}_{n}-\psi^{+}_{\infty}),\psi^{+}_{n}-\psi^{+}_{\infty}\rangle=o(1).
\ee
An analogous argument gives 
\be
\Vert\psi^{-}_{n}-\psi^{-}_{\infty}\Vert^{2}=o(1)
\ee
thus proving the pre-compactness of Cerami sequences.
\end{proof}
Our main theorem (Theorem \ref{main}) is thus proved, as the regularity of solutions follows by standard bootstrap techniques, exploiting the regularization property of $(-\Delta)^{-\frac{1}{2}}$. 
\section{Auxiliary results}\label{appendix}
This section contains some auxiliary results used in the proof of our main theorem.

\noindent\textbf{Compactness of dF.}
\begin{lemma}\label{brezis}
Let $(Y,\Vert\cdot\Vert_{Y})$ be a uniformly convex Banach space and consider a sequence $(y_{n})_{n\in\mathbb{N}}\subseteq Y$. Suppose that $y_{n}\rightharpoonup y$ weakly in $Y$ and $\Vert y_{n}\Vert_{Y}\rightarrow\Vert y\Vert_{Y}$. 

Then $y_{n}\longrightarrow y$ strongly in $Y$, as $n\rightarrow+\infty$.
\end{lemma}
See \cite{functional} for a proof. The above lemma allows us to prove the following
\begin{prop}\label{convergence}
Let $(\psi_{n})_{n\in\mathbb{N}}\subseteq X$ be a sequence such that $\psi_{n}\rightarrow \psi\in X$ strongly in $L^{p}$, for all $1\leq p< 4$.

Then, up to a subsequence $\vert\psi_{n}\vert^{2}\rightarrow\vert\psi\vert^{2}$ strongly in $L^{\frac{3}{2}}$, as $n\rightarrow+\infty$.
\end{prop}
\begin{proof}
We have 
\be
\Vert\vert\psi_{n}\vert^{2}\Vert_{\frac{3}{2}}=\Vert\psi_{n}\Vert^{2}_{3}\longrightarrow\Vert\psi\Vert^{2}_{3}=\Vert\vert\psi\vert^{2}\Vert_{\frac{3}{2}}
\ee
as $n\rightarrow+\infty$, since $\psi_{n}\rightarrow\psi$ strongly in $L^{3}$.

Moreover, it is easy to see that 
\be
\Vert\vert\psi_{n}\vert^{2}\Vert_{\frac{3}{2}}=\Vert\psi_{n}\Vert^{2}_{3}\leq C
\ee
and thus, up to a subsequence, $\vert\psi_{n}\vert^{2}\rightharpoonup\vert\psi\vert^{2}$ weakly in $L^{\frac{3}{2}}$.

Then the claim follows by Lemma (\ref{brezis}), $L^{p}$ spaces being uniformly convex for $1<p<+\infty$ (see,e.g. \cite{functional}).
.\end{proof}

Consider the map $F:X\longrightarrow \R$, defined as 
\be 
F(\psi):=\frac{1}{4}\int_{\Omega}\cV(\psi)\vert\psi\vert^{2}
\ee
It can be easily seen that the differential $dF:X\longrightarrow X^{*}$ acts as follows:
\be
\langle dF(\psi),\varphi\rangle_{X^{*}\times X}=\int_{\Omega}\cV(\psi)\Re(\psi\overline{\varphi}),\qquad \forall \psi,\varphi\in X,
\ee
where $\Re(\cdot)$ denotes the real part of a complex number.
\begin{prop}\label{homogeneous}
The differential $dF$ is compact.
\end{prop}
\begin{proof}
Let $(\psi_{n})_{n\in\mathbb{N}}\subseteq X$ be a bounded sequence. Then the compactness of the Sobolev embedding $H^{\frac{1}{2}}(\Omega,\C^{2})\hookrightarrow L^{p}(\Omega,\C^{2})$, for $1\leq p<4$, implies that, up to subsequences, $\psi_{n}\rightarrow\psi\in X$, strongly in $L^{p}$. 

Take $\varphi\in X$ with $\Vert\varphi\Vert\leq 1$. We then have
\be\label{compact}
\begin{split}
\left\vert\int_{\Omega}\left(\cV(\psi_{n})\psi_{n}-\cV(\psi)\psi\right)\overline{\varphi} \right\vert&\leq \int_{\Omega}\left\vert (\cV(\psi_{n})\psi_{n}-\cV(\psi_{n})\psi)\overline{\varphi} \right\vert \\&+\int_{\Omega}\left\vert (\cV(\psi_{n})\psi-\cV(\psi)\psi)\overline{\varphi} \right\vert
\end{split}
\ee
We estimate the first term in the r.h.s. as follows. 

Applying the Cauchy-Schwarz inequality we get
\be\label{uno}
\begin{split}
\int_{\Omega}\left\vert (\cV(\psi_{n})\psi_{n}-\cV(\psi_{n})\psi)\overline{\varphi} \right\vert&\leq\left(\int_{\Omega}\vert \cV(\psi_{n})\varphi\vert^{2}\right)^{\frac{1}{2}}\left(\int_{\Omega}\vert\psi_{n}-\psi\vert^{2}\right)^{\frac{1}{2}}\\&\leq\Vert\varphi\Vert_{4}\Vert \cV(\psi_{n})\Vert_{4}\Vert\psi_{n}-\psi\Vert_{2}\\&\leq C\Vert\psi_{n}-\psi\Vert_{2}\longrightarrow0
\end{split}
\ee
for $n\rightarrow+\infty$, using the fact that $(-\Delta)^{-\frac{1}{2}}$ maps continuously $L^{2}$ to $H^{1}\hookrightarrow L^{4}$, and $\cV(\psi)=(-\Delta)^{-\frac{1}{2}}\vert\psi\vert^{2} $.

For the second term in (\ref{compact}), we use again the Cauchy-Schwarz inequality and get
\be\label{due}
\begin{split}
\int_{\Omega}\left\vert (\cV(\psi_{n})\psi-\cV(\psi)\psi)\varphi \right\vert&\leq \left(\int_{\Omega}\vert \psi\varphi\vert^{2}\right)^{\frac{1}{2}}\left(\int_{\Omega}\vert \cV(\psi_{n})-\cV(\psi)\vert^{2}\right)^{\frac{1}{2}}\\&\leq\Vert\varphi\Vert_{4}\Vert \psi\Vert_{4}\Vert \cV(\psi_{n})-\cV(\psi)\Vert_{2}\\&\leq C\Vert \cV(\psi_{n})-\cV(\psi)\Vert_{2} \longrightarrow 0
\end{split}
\ee
as $n\rightarrow+\infty$, since  $\cV(\psi_{n})=(-\Delta)^{-\frac{1}{2}}\vert\psi_{n}\vert^{2}\in W^{1,\frac{3}{2}}(\Omega,\C^{2})\hookrightarrow L^{2}(\Omega,\C^{2})$ and $\vert\psi_{n}\vert^{2}\rightarrow\vert\psi\vert^{2}$ strongly in $L^{\frac{3}{2}}$, as shown in Prop.(\ref{convergence}).

Thus combining (\ref{uno}) and (\ref{due}) we have
\be
\left\vert\int_{\Omega}\left(\cV(\psi_{n})\psi_{n}-\cV(\psi)\psi\right)\overline{\varphi} \right\vert\longrightarrow0
\ee
uniformly with respect to $\varphi$, as $n\rightarrow+\infty$.
\end{proof}

\noindent\textbf{$(-\Delta)^{-\frac{1}{2}}$ is positivity-preserving.}
For the sake of brevity we will only sketch the argument, referring to the mentioned references for more details.

Recall that $-\Delta$ is the Dirichlet laplacian on $L^{2}(\Omega)$, with domain $H^{1}_{0}(\Omega)$.

The starting point is the following identity of $L^{2}$-operators: 
\be\label{formula}
(-\Delta)^{-\frac{1}{2}}=\frac{1}{\sqrt{\pi}}\int^{+\infty}_{0}e^{-t^{2}\Delta} dt.
\ee 
Indeed, let $(e_{n})_{n\in\mathbb{N}}\subseteq L^{2}(\Omega)$ be a Hilbert basis of eigenfunctions of $-\Delta$, with associated eigenvalues $0<\mu_{n}\uparrow+\infty$. 

 For any $n\in\mathbb{N}$ the operator on the r.h.s. of (\ref{formula}) acts on each $e_{n}$ as the multiplication operator by the function
\be
\frac{1}{\sqrt{\pi}}\int^{+\infty}_{0}e^{-t^{2}\mu_{n}} dt=\frac{1}{\sqrt{\mu_{n}}}.
\ee

To prove the claim it is thus sufficient to prove that the heat kernel $e^{-s\Delta}$ is positivity-preserving. This follows from the
\begin{thm}{(First Beurling-Deny criterion)}\label{criterion}
Let $L\geq 0$ be a self-adjoint operator on $L^{2}(\Omega)$. Extend $\langle u,Lu\rangle_{L^{2}}$ to all $L^{2}$ by setting it equal to $+\infty$, when $u$ does not belong to the form-domain of $L$. The following are equivalent:

\begin{itemize}
\item $e^{-sL}$ is positivity-preserving for all $s>0$;
\item $\langle\vert u\vert,L\vert u\vert\rangle_{L^{2}}\leq\langle u,Lu\rangle_{L^{2}}, \quad\forall u\in L^{2}(\Omega)$.
\end{itemize}

\end{thm}
A proof of the above result can be found in (\cite{reedsimonIV}, Theorem XIII.50).

Taking $L=-\Delta$, the second condition in the above theorem corresponds to the well-known fact that for any $u\in H^{1}_{0}(\Omega)$ there holds $$\vert\nabla \vert u\vert\vert\leq\vert\nabla u\vert\qquad\mbox{a.e. in}\quad\Omega$$  and then
$$\langle\vert u\vert,L\vert u\vert\rangle_{L^{2}}=\int_{\Omega}\vert\nabla\vert u\vert\vert^{2}\leq\int_{\Omega}\vert\nabla u\vert^{2}=\langle u,Lu\rangle_{L^{2}} $$
(see Theorem 6.1 in \cite{liebloss}). This concludes the proof.

\noindent\textbf{Acknowledgments.} 
I wish to thank \'{E}ric S\'{e}r\'{e} for helpful conversations on his work on Maxwell-Dirac equations and Mathieu Lewin for pointing out useful references.

\bibliographystyle{siam}
\bibliography{DiracHartree}

\end{document}